\documentclass[a4paper, 11pt]{article}

\usepackage{amsmath, amssymb,amscd, amsthm}
\usepackage[mathscr]{eucal}
\usepackage{graphics}
\usepackage{fullpage}
\usepackage{url}
\newcommand\cyr{%
 \renewcommand\rmdefault{wncyr}%
 \renewcommand\sfdefault{wncyss}%
 \renewcommand\encodingdefault{OT2}%
\normalfont\selectfont} \DeclareTextFontCommand{\textcyr}{\cyr}

\newtheorem{theorem}{Theorem}

\newtheorem{corollary}[theorem]{Corollary}
\newtheorem{proposition}[theorem]{Proposition}
\newtheorem{remark}[theorem]{Remark}

\def\Z{\mathbb Z}

\def\mod{\operatorname{mod}}

\begin{document}

\title{\textbf{Maximal Operators Associated to Multiplicative Characters}}

\author{Allison Lewko\thanks{Work completed while supported by a Microsoft Research PhD Fellowship.} \and Mark Lewko \thanks{Supported by a NSF Postdoctoral Fellowship, DMS-1204206}}


\date{}

\maketitle

\begin{abstract}We show that the natural analogue of the Carleson-Hunt inequality fails for multiplicative characters.
\end{abstract}

\section{Introduction}The Carleson-Hunt inequality states that there exists a finite constant $C$ such that

\begin{equation}\label{CH1}
\left( \int_{0}^{1} \left| \max_{\ell}  \left|\sum_{n=1}^{\ell} a_n e(x n) \right| \right|^2 dx \right)^{1/2} \leq C \left( \sum_{n=1}^{\infty} |a_n|^2 \right)^{1/2}
\end{equation}
for any sequence of complex numbers $\{a_n\}_{n=1}^{\infty}$ (denoting  $e(x):= e^{2 \pi i x}$).  It is not hard to see that this is equivalent to the discretized claim that

\begin{equation}\label{CH}
 \left( \frac{1}{N}\sum_{x=1}^{N} \left| \max_{\ell \leq N} \left|\sum_{n=1}^{\ell} a_n e( x n/ N) \right| \right|^2 \right)^{1/2} \leq C \left( \sum_{n=1}^{N} |a_n|^2 \right)^{1/2}
 \end{equation}
holds with a universal constant $C$, independent of $N$. This can be viewed as a natural analogue of the Carleson-Hunt inequality in the family of additive groups, $\mathbb{Z}_N$. The fact that this inequality implies (\ref{CH1}) follows from an easy approximation argument. The reverse implication is slightly more subtle but can be obtained, for instance, from Montgomery's maximal large sieve inequality \cite{M} (Theorem 2).

In light of (\ref{CH}), it is natural to consider the analogous maximal operator on the family of multiplicative groups, $\mathbb{Z}_N^{*}$. Similar multiplicative quantities arise naturally in sieve theory (this connection was our principle motivation and is briefly discussed at the end of this note). That is, we consider the inequality

\begin{equation}\label{DirichletMax}
\left( \frac{1}{\phi(N)}\sum_{\chi \mod N} \left| \max_{\ell < N} \left|\sum_{\substack{n=1 \\ (n,N)=1 }}^{\ell} a_n \chi(n) \right|\right|^2 \right)^{1/2} \leq \Delta(N) \left(\sum_{\substack{n=1 \\ (n,N)=1 }}^{N-1} |a_n|^2 \right)^{1/2}
\end{equation}
where the sum $\sum_{\chi \mod N}$  is over all Dirichlet characters modulo $N$, $\phi(N)$ is the Euler totient function, and $\Delta(N)$ is the smallest value such that the inequality holds. We are interested in the growth of the function $\Delta(N)$. It is easy to see that $\Delta(N) \ll \log(N)$ by the Rademacher-Menshov theorem (see Section \ref{sec:remarks}). By comparison to (\ref{CH}), one might hope that $\Delta(N) \ll 1$. In fact, we show in Section 2 that if one could take $\Delta(N)=O(1)$ in (\ref{DirichletMax}), then the Carleson-Hunt inequality for the trigonometric system would be an easy corollary. We note that in the case that $a_n=1$ for all $n$ it follows from work of Montgomery and Vaughan \cite{MV} that (\ref{DirichletMax}) holds with a universal constant independent of $N$.

Unfortunately, it turns out that $\Delta(N) \neq O(1)$. We prove:

\begin{theorem}\label{mainThm} There exists a subset $\mathcal{S}$ of primes of positive relative density (within the primes) such that for every $p \in \mathcal{S}$,
\begin{equation}
(\log \log (p))^{1/4} \ll \Delta(p).
\end{equation}
\end{theorem}
Thus $(\log \log (N))^{1/4} \ll \Delta(N)$ holds infinitely often. It remains an interesting problem to establish sharp bounds on the growth of $\Delta(N)$. In particular, any refinement of the upper bound $\Delta(N) \ll \log(N)$ (from the Rademacher-Menshov theorem) would be extremely interesting.

Let us briefly describe the idea behind the proof of Theorem \ref{mainThm}. It is known (Kolmogorov's rearrangement theorem) that if one replaces the maximal operator $\max_{\ell}  \left|\sum_{n=1}^{\ell} a_n e(x n) \right| $ with an operator of the form $\max_{\ell}  \left|\sum_{n=1}^{\ell} a_n e(x \sigma(n)) \right| $ for a carefully chosen permutation $\sigma : \mathbb{Z}_{+} \rightarrow \mathbb{Z}_{+}$ then one can force the analogue of (\ref{CH1}) to fail. Analogously, one can find a family of permutations of $[N]$ (for increasing $N$) that forces the constant in (\ref{CH}) to grow with $N$ (see Corollary \ref{NakataCor} below for a precise formulation). We will show that if one restricts the maximal operator associated to multiplicative group $\mathbb{Z}_{p}^{*}$ (for suitable $p$) to a carefully chosen subset of the coefficients, then this operator is equivalent to the (additive) maximal operator on a the group $\mathbb{Z}_{q}$ in a badly behaved ordering, for a related (but much smaller) $q$.

\section{Connection with the Carleson-Hunt inequality}

In this section, we prove that if (\ref{DirichletMax}) with $\Delta(N)=O(1)$ did hold then the classical Carleson-Hunt inequality (\ref{CH1}) would easily follow. By a standard density argument, it suffices to prove (\ref{CH1}) for finite sequences $\{a_n\}_{n=1}^{k}$ as long as the constant $C$ does not depend on $k$. Indeed, once $k$ is fixed, we will show that

\begin{equation}
 \left( \frac{1}{M}\sum_{x=1}^{M} \left| \max_{\ell < k} \left|\sum_{n=1}^{\ell} a_n e( x n/ M) \right| \right|^2 \right)^{1/2} \leq C \left( \sum_{n=1}^{k} |a_n|^2 \right)^{1/2}
\end{equation}
holds for an infinite increasing sequence of integral $M$'s and a constant $C$ independent of $k$. Clearly, this is sufficient as the sum on the left will converge to the Riemann integral over the unit interval. Consider a large prime $2^{k} < p$. Now we apply (\ref{DirichletMax}) with $\Delta(p)\leq C$, and associating the coefficient $a_i$ to $\chi(2^i)$ we have

\begin{equation}\label{DirichletMax5}
\left( \frac{1}{\phi(p)}\sum_{\chi \mod p} \left| \max_{\ell < k } \left|\sum_{\substack{i=1}}^{\ell} a_i \chi(2^i) \right|\right|^2 \right)^{1/2} \leq C \left(\sum_{n=1}^{k} |a_n|^2 \right)^{1/2}.
\end{equation}

We let $\alpha$ be a generator of $\Z_p^*$. For $g \in \Z_p^*$, we define $\nu(g)$ to be the element of $[p-1]$ such that $\alpha^{\nu(g)}=g$.
We may express $\chi(g)=e(a \nu(g) / (p-1))$ (for some $0 \leq a < p-1 $). Thus $\nu$ can be thought of as a permutation of $[p-1]$. In addition, it follows from this definition that $\nu(2^i)=i\nu(2)$ (mod p). We thus have that (\ref{DirichletMax5}) can be expressed as:

\begin{equation}\label{DirichletMax6}
\left( \frac{1}{p-1} \sum_{x=1}^{p-1} \left| \max_{\ell < k } \left|\sum_{\substack{i=1}}^{\ell} a_i e\left(\frac{i\nu(2)x}{p-1} \right) \right|\right|^2 \right)^{1/2} \leq C \left(\sum_{n=1}^{k} |a_k|^2 \right)^{1/2}.
\end{equation}

We define coprime $L,M \in \Z$ by $\frac{\nu(2)}{p-1} = \frac{L}{M}$. We observe that $M > \log_2(p)$. To see this, consider that $M \nu(2) = L (p-1)$ implies $2^M = \alpha^{L (p-1)}\equiv 1$. Note that $M | (p-1)$. Substituting this into (\ref{DirichletMax6}), we have
\begin{equation}\label{DirichletMax7}
\left( \frac{1}{p-1} \sum_{x=1}^{p-1} \left| \max_{\ell < k } \left|\sum_{\substack{i=1}}^{\ell} a_i e\left(\frac{iLx}{M} \right) \right|\right|^2 \right)^{1/2} \leq C \left(\sum_{n=1}^{k} |a_k|^2 \right)^{1/2}.
\end{equation}

Since $M | (p-1)$ and $e\left(\frac{iLx}{M} \right)$ has period $M$ as a function of $x$, this can be rewritten as:
\begin{equation}\label{DirichletMax8}
\left( \frac{1}{M} \sum_{x=1}^{M} \left| \max_{\ell < k } \left|\sum_{\substack{i=1}}^{\ell} a_i e\left(\frac{iLx}{M} \right) \right|\right|^2 \right)^{1/2} \leq C \left(\sum_{n=1}^{k} |a_k|^2 \right)^{1/2}.
\end{equation}
We perform the change of variable $Lx \rightarrow y$ to obtain
\begin{equation}\label{DirichletMax9}
\left( \frac{1}{M} \sum_{y=1}^{M} \left| \max_{\ell < k } \left|\sum_{\substack{i=1}}^{\ell} a_i e\left(\frac{iy}{M} \right) \right|\right|^2 \right)^{1/2} \leq C \left(\sum_{n=1}^{k} |a_k|^2 \right)^{1/2}.
\end{equation}
for some $M \gg \log(p)$. This completes the proof.

\begin{remark}One could also deduce the Carleson-Hunt inequality for Walsh series from the claim that $\Delta(N)=O(1)$. We briefly sketch the argument. Choose $N$ to be the product of $d$ distinct odd primes. Then $Z_N^{*}$ will contain an isomorphic copy of the group $Z_{2}^{d}$. The characters of this group are distributionally equivalent with the first $2^d$ Walsh functions. The maximal operator on $\mathbb{Z}_N^{*}$ will induce some ordering on these functions other than the standard ordering. It follows, however, from a combinatorial lemma of Bourgain \cite{B} (Lemma 2.3) that there is a function $B(d)$ (tending to infinity) such that any ordering of the first $2^d$ Walsh functions must contain a subsequence of length $B(d)$ distributionally equivalent to the first $B(d)$ Walsh functions in the standard ordering.
\end{remark}

\section{Auxiliary Results}

In this section, we collect some auxiliary results that will be needed in the proof of Theorem \ref{mainThm}.
We first note the following result from \cite{F}:
\begin{proposition}\label{primes} (Fouvry)
Let $\mathcal{P}(N)$ denote the largest prime divisor of $N$. Then for a positive
proportion of the primes $p$, we have $\mathcal{P}(p-1) \geq  B p^{.6687}$ for some positive constant $B$.
\end{proposition}
For our purposes, $\mathcal{P}(p-1) \gg p^{\frac{1}{2} + \epsilon}$ for any fixed $\epsilon>0$ would suffice.
We will need a quantitative multi-dimensional form of Weyl's criterion which can be found in Chapter 2 of \cite{KN}:

\begin{proposition}\label{ETK}(Erd\"{o}s-Turan-Koksma) Let $P$ denote a sequence of $N$ points, $x_1,x_2,\ldots,x_N \in [0,1]^s$. Define the discrepancy of this sequence as

\[D_N(x_1,x_2,\ldots,x_N) :=  \sup_{I \in \mathcal{B}} \left| \frac{ |I\cap P|}{N} - |I| \right|,  \]
where $\mathcal{B}$ denotes the set of all $s$-dimensional boxes in $[0,1]^s$, $|I|$ denotes the Lebesgue measure of $I$, and $|I \cap P|$ denotes the number of points $x_i \in I$. Furthermore, for $h \in \mathbb{Z}^{s}$ let
\[r(h) :=\prod_{i=1}^{s} \max(1,|h_i|).\]
Then, for all $m \in \mathbb{N}$, we have that

\begin{equation}
D_N(x_1,x_2,\ldots,x_N) \leq 2s^2 3^{s+1} \left(\frac{1}{m} + \sum_{\substack{h \in \mathbb{Z}^s \\ 0 < ||h||_{\infty}\leq m   } } \frac{1}{r(h)} \left| \frac{1}{N} \sum_{n=1}^{N} e(\left< h, x_n \right>) \right| \right),
\end{equation}
where $||h||_\infty$ denotes the maximal element of $h$ in absolute value, and $\left< h, x_n \right>$ denotes the dot product over $\mathbb{R}$.

\end{proposition}

We will also need the following version of Weil's character sum estimate. This can be found, for instance, on page 45 of \cite{S}:

\begin{proposition}\label{Weil}(Weil) Let $p$ be a prime and $g(x)= g_n x^n+ \ldots +g_0$ a degree $n$ polynomial ($0<n<p$) with integer coefficients such that $p \nmid g_n$. Then,

\begin{equation}\label{Weq}
\left| \sum_{x=0}^{p-1} e(g(x)/p) \right| \leq (n-1) p^{1/2}.
\end{equation}

\end{proposition}

We will also use the following quantitative form of Kolmogorov's rearrangement theorem due to Nakata \cite{N} (Lemma 4):

\begin{proposition}\label{Nakata1}(Nakata) There exist universal real constants $c_1, c_2 > 0$ with the following property. For any $N \in \mathbb{N}$, there exists a permutation $\sigma:[N]\rightarrow [N]$ (where $[N]$ denotes the set $\{1, 2, \ldots, N\}$) and complex numbers $\{a_n\}_{n=1}^{N}$ satisfying $\sum_{n=1}^{N}|a_n|^2=1$ such that

\[ \left| \{x \in [0,1] :  \widetilde{\mathcal{M}}(x) > c_1 \log^{1/4}(N) \} \right|   \geq c_2 \]
holds, where

\[ \widetilde{\mathcal{M}}(x):= \max_{\ell \leq N} \left|\sum_{n=1}^{\ell} a_n e(\sigma(n) x)\right| .\]

\end{proposition}

We remark that Nakata has a slightly stronger refinement of Kolmogorov's rearrangement theorem \cite{N2} where there are some additional iterated logarithmic factors. However, the result there is formulated in a slightly different way and it would require some additional work to derive a statement sufficient for our purposes from it. For the sake of simplicity, we will not pursue this modification here. Next, we derive a discrete version of Proposition \ref{Nakata1}.

\begin{corollary}\label{NakataCor}There exist universal real constants $c_1, c_2 > 0$ with the following property. For any $N \in \mathbb{N}$, there exists a permutation $\sigma:[N]\rightarrow [N]$ and complex numbers $\{b_n\}_{n=1}^{N}$ satisfying $\sum_{n=1}^{N}|b_n|^2=1$ such that

\[ \left| \{a \in [M] :  \mathcal{M}(a) > c_1 \log^{1/4}(N) \} \right|   \geq c_2 M \]
holds for any positive integer $M$, where $\mathcal{M}:[M] \rightarrow \mathbb{R}$ is defined by

\[ \mathcal{M}(a):= \max_{\ell \leq N} \left|\sum_{n=1}^{\ell} b_n e(\sigma(n) a/M)\right| .\]

\end{corollary}
\begin{proof}Consider
\[\widetilde{\mathcal{M}_{\tau}}(x) := \max_{\ell \leq N} \left|\sum_{n=1}^{\ell} a_n e(\sigma(n)\tau ) e(\sigma(n) x)\right| = \widetilde{\mathcal{M}}(x+\tau),\]
where $\{a_n\}_{n=1}^N$ and $\sigma$ are given by Proposition \ref{Nakata1}.

It suffices to show that for some $\tau \in [0,1]$ we have
\[\left| \{a \in [M] :  \widetilde{\mathcal{M}_{\tau}}(a/M) > c_1 \log^{1/4}(N) \} \right|   \geq c_2 M \]
as the corollary then holds with $b_n = e(\sigma(n) \tau) a_n$. Let   $U:=\{x \in [0,1] :  \widetilde{\mathcal{M}}(x) > c_1 \log^{1/4}(N) \}$ be the set considered by Proposition \ref{Nakata1}. Denoting the characteristic function of $U$ as $\mathbb{I}_{U}$, we then have that
\[ \int_{0}^{1/M} \left| \{a \in [M] :  \widetilde{\mathcal{M}_{\tau}}(a/M) > c_1 \log^{1/4}(N) \} \right|d\tau \]
\[= \int_{0}^{1/M} \sum_{a=1}^{M} \mathbb{I}_{U}(a/M+\tau) d\tau = \int_{0}^{1} \mathbb{I}_{U}(x)dx  \geq c_2,\]
where the last inequality follows from Proposition \ref{Nakata1}. Thus for some $\tau_{0} \in [0,1/M]$ we must have that $ \left| \{a \in [M] :  \widetilde{\mathcal{M}_{\tau_{0}}}(a/M) > c_1 \log^{1/4}(N) \} \right| \geq c_2 M$.
\end{proof}

From these results, we obtain:

\begin{proposition}\label{mainProp}We denote the fractional part of $a \in \mathbb{R}$ by $\{a\}$. We let $p,q$ denote primes such that $q|p-1$ and $q \geq B p^{.6687}$. We let $\mathcal{A}$ denote the subgroup of order $q$ in $\Z_p^*$ (i.e. $\mathcal{A}= \left\{g^{\frac{p-1}{q}}: g \in \Z_p^*\right\}$).
There exists a universal constant $\delta >0$ such that for any $s < \delta \log^{1/2}(p)$, and any permutation $\sigma:[s]\rightarrow [s]$, there exists an $x \in \mathcal{A}$ such that
\[ \left\{\frac{x^{\sigma(1)}}{p}\right\} <   \left\{\frac{x^{\sigma(2)}}{p}\right\} < \ldots  < \left\{\frac{x^{\sigma(s)}}{p}\right\}.\]
\end{proposition}

\begin{proof} We let $g_1, \ldots, g_q$ denote the elements of $\mathcal{A}$ inside $\Z_p^*$ in the order induced by $\Z_p^*$ (that is the order of the representative integers of $\Z_p^*$ between $0$ and $p-1$). For each $i$ from 1 to $q$, we define $y_i \in [0,1]^s$ as $y_{i}= \left(\left\{\frac{g_i^{1}}{p}\right\},\left\{\frac{g_i^{2}}{p}\right\},\ldots,\left\{\frac{g_i^{s}}{p}\right\} \right)$. We then divide $[0,1]^s$ into $(3s)^s$ boxes of equal measure by dividing each coordinate into $3s$ equal intervals in the obvious way. The conclusion will now follow if we show that there exists a point $y_i$ in each of the $(3s)^s$ boxes.
To see this, consider the $3s$ intervals in each coordinate as being $s$ groups of 3 intervals each, and let $I^k_j$ denote the ``middle" interval of the $j^{th}$ group in the $k^{th}$ coordinate. Note that $I^k_j$ and $I^k_{j'}$ for $j \neq j'$ do not intersect. Given a permutation $\sigma$, it suffices to obtain a point in the box whose $k^{th}$ side is equal to $I^k_{\sigma(k)}$. (Thus, it actually suffices to show the existence of a point $y_i$ in every such ``middle" box, but we will prove the stronger statement that there is a $y_i$ in every box.)

To establish the existence of a point $y_i$ in every box, it suffices to show $D(y_1,y_2,\ldots,y_q) < (3s)^{-s}$. Invoking Proposition \ref{ETK}, we have that
\begin{equation}\label{DiscIn1}
D(y_1,y_2,\ldots,y_q) \leq 2s^2 3^{s+1} \left(\frac{1}{m} + \sum_{0 < ||h||_{\infty}\leq m} \frac{1}{r(h)} \left| \frac{1}{q} \sum_{i=1}^{q} e(\left< h, y_i \right>) \right| \right).
\end{equation}

We define
\[f_h(x) =  h_1 x^{\frac{p-1}{q}} + h_2 x^{2 \frac{p-1}{q}}  + \ldots + h_s x^{s \frac{p-1}{q}} .\]

Noting that the map $x \rightarrow x^{\frac{p-1}{q}}$ sends $\mathbb{Z}_{p}^{*}$ onto $\mathcal{A}$ with each point in the range having exactly $\frac{p-1}{q}$ preimages, we then have:
\[\frac{1}{q} \sum_{i=1}^q e(\left< h, y_i\right>)= \frac{1}{q}  \sum_{g \in \mathcal{A} } e((h_1 g + h_2g^2 + \ldots + h_s g^s) /p )  \]
\[=  \frac{1}{q} \left( \frac{q}{p-1} \sum_{x \in \mathbb{Z}_{p}^{*} } e((h_1 x^{\frac{p-1}{q}} + h_2 x^{2 \frac{p-1}{q}} + \ldots + h_s x^{s\frac{p-1}{q}}) /p ) \right)\]
\[ =  \frac{1}{p-1}\left( -1 + \sum_{x=0}^{p-1} e(f_h(x)/p) \right).\]

We note that  \[\sum_{0 < ||h||_{\infty}\leq m} \frac{1}{r(h)} = \left(1 + \sum_{j=1}^{m}j^{-1}\right)^{s} \leq C^s\log^s(m)\] for some constant $C$. Whenever some $h_i$ is not divisible by $p$, we may apply Proposition \ref{Weil} to bound the quantity $\left|\sum_{x=0}^{p-1} e(f_h(x)/p)\right|$. We will choose $m < p$ so that all $h$'s will have this property. We may thus bound the right hand side of (\ref{DiscIn1}) by
\begin{equation}\label{eq:bound}
 \leq 2s^2 3^{s+1} \left(\frac{1}{m} + s p^{-.1687} C^s\log^s(m) \right) ,
\end{equation}
when $m < p$ (for some new value of $C$). Here, we have applied Proposition $\ref{Weil}$ to polynomials of degree $\leq s\left(\frac{p-1}{q}\right)$.

For any constant $\delta_2>0$, we can set $m= s^{\delta_1 s}$ for some constant $\delta_1$ sufficiently large so that (\ref{eq:bound}) is
\begin{equation}\label{eq:bound2}
\leq s^{-\delta_2 s} +  p^{-.1687} s C^s (\delta_1 s \log(s) )^{s}.
\end{equation}
Fixing $\delta_2$ such that $s^{-\delta_2 s} \leq \frac{1}{2} (3s)^{-s}$ for all $s >1$ say (note that the Proposition is trivial for $s=1$), we may then require that $s$ satisfy $ p \geq s^{\delta_3 s} $ for $\delta_3$ sufficiently large so that $\delta_3 > \delta_1$ and the quantity in (\ref{eq:bound2}) is $< (3s)^{-s} $. We observe that $s ^{\delta_3 s} \leq p$ is equivalent to $\delta_3 s \log(s) \leq \log(p)$, which can be guaranteed by $s \leq \delta \log^{1/2}(p)$ for a suitable choice of $\delta$.
\end{proof}

\section{Proof of Theorem 1}
We define $\mathcal{S}$ to be the set of primes $p$ such that there exists a prime $q$ dividing $p-1$ with $q \geq Bp^{.6687}$. For each such $p$, we fix such a $q$. By Proposition \ref{primes}, this is an infinite set of positive relative density in the primes.
For each $p \in \mathcal{S}$, we let $\mathcal{A}$ denote the subgroup of order $q$ in $\Z_p^*$.

Our goal is to define suitable coefficients supported on $\mathcal{A}$ to show that $\Delta(p)$ is $\gg (\log\log(p))^{\frac{1}{4}}$.

We enumerate the elements of $\mathcal{A}$ in the natural way (that is so their smallest representatives in $\mathbb{Z}_{+}$ are ordered in increasing order), say $\{g_n\}_{n=1}^{q}$. Next, we let $\alpha$ be a generator of $\mathcal{A}$. We define $\nu(g_n)$ to be the element of $[q]$ such that $\alpha^{\nu(g_n)}=g_n$.
By restricting the coefficients in (\ref{DirichletMax}) to $\mathcal{A}$, we see that the quantity in (\ref{DirichletMax}) to be bounded is:
\begin{equation}\label{DirichletMax2}
\left( \frac{1}{p-1}\sum_{\chi \mod p} \max_{\ell \leq q} \left| \sum_{n=1}^{\ell} a_n \chi(g_n)\right|^2\right)^{\frac{1}{2}} = \left(\frac{1}{q} \sum_{x=1}^q  \max_{\ell \leq q} \left| \sum_{n=1}^\ell a_n e(\nu(g_n) x /q)\right|^2\right)^{\frac{1}{2}}.
\end{equation}
This follows because restricting $\chi \mod p$ to $\mathcal{A}$ yields a character on $\mathcal{A}$.

Let $ s = \lfloor \delta \log^{1/2}(p)\rfloor$ and $\sigma: [s] \rightarrow [s]$ be the permutation in Corollary \ref{NakataCor}, along with coefficients $b_1, \ldots, b_s$ such that $\sum_{m=1}^s |b_m|^2 = 1$. By Proposition \ref{mainProp}, we have a $g \in \mathcal{A}$ such that
\[ \left\{\frac{g^{\sigma(1)}}{p}\right\} <   \left\{\frac{g^{\sigma(2)}}{p}\right\} < \ldots  < \left\{\frac{g^{\sigma(s)}}{p}\right\}.\]
Of course $g^{\sigma(1)},g^{\sigma(2)},\ldots,g^{\sigma(s)} \in \mathcal{A}$. It follows that the corresponding terms in the inner sum on the right of (\ref{DirichletMax2}) appear in exactly this order. By restricting the support of our coefficients to these terms in (\ref{DirichletMax2}) and using $b_1, \ldots, b_s$ as our coefficients, it suffices to consider the quantity
\[ \left(\frac{1}{q} \sum_{x=1}^{q} \left| \max_{\ell < s} \left|\sum_{m=0}^{\ell} b_m e(\nu(g^{\sigma(m)})x/q)  \right|\right|^2 \right)^{1/2} \]
\begin{equation}\label{DirichletMax3}
 = \left(\frac{1}{q} \sum_{x=1}^{q} \left| \max_{\ell < s} \left|\sum_{m=0}^{\ell} b_m e(\sigma(m)\nu(g)x/q)  \right|\right|^2 \right)^{1/2}
\end{equation}
where we have exploited the fact that $\nu(g^{i})=i\nu(g)$ (mod p). Finally, by the change of variables $x\nu(g) \rightarrow y$, we have
\begin{equation}\label{DirichletMax4}
(\ref{DirichletMax3})= \left(\frac{1}{q} \sum_{y=1}^{q} \left| \max_{\ell \leq s} \left|\sum_{m=1}^{\ell} b_m e(\sigma(m)y/q)  \right|\right|^2 \right)^{1/2}.
\end{equation}

Applying Corollary \ref{NakataCor} with $M=q$, we see that this quantity is $\gg (\log(s))^{1/4} \gg (\log(\log(p)))^{1/4}$. This completes the proof.

\section{Remarks}\label{sec:remarks}

\begin{itemize}
  \item We remark on how to derive the $\log(N)$ upper bound from the Rademacher-Menshov theorem. Let $\{\phi_n\}_{n=1}^{N}$ denote a collection of orthonormal functions on a probability space $\mathbb{T}$. The Rademacher-Menshov theorem states that
      \[\left( \int_{\mathbb{T}} \left| \max_{\ell \leq N}  \left|\sum_{n=0}^{\ell} a_n \phi_n(x) \right| \right|^2 dx \right)^{1/2} \ll \log(N) \left( \sum_{n=0}^{\infty} |a_n|^2 \right)^{1/2},  \]
where the $\log(N)$ is known to be necessary (for some choices of orthonormal systems). Let $\mathbb{T}$ denote the probability space of the Dirichlet characters mod $N$ with normalized counting measure. Now it is easy to see that the $\phi(N)$ functions $\delta_n(\chi):=\chi(n)$ form an orthonormal system on $\mathbb{T}$. Thus the Rademacher-Menshov theorem gives $\Delta(N) \ll \log(\phi(N)) \ll \log(N)$.

   \item There is some flexibility in the techniques applied in the proof of Theorem \ref{mainThm}, and variants of these arguments should give lower bounds on $\Delta(N)$ for some more general $N$. However, a more delicate analysis will be needed to obtain a uniform lower bound in $N$.
  \item It is consistent with our knowledge that one might be able to replace the $\log^{1/4}(N)$ in Proposition \ref{Nakata1} with a $\log(N)$. This would  allow one to strengthen the conclusion of Theorem \ref{mainThm} to $\log\log(p) \ll \Delta(p)$. However, the Rademacher-Menshov theorem prevents the conclusion of Proposition \ref{Nakata1} from holding with any function growing faster than $\log(N)$. Thus a lower bound of $\log\log(N)$ would be the limitation of the approach developed here.
  \item One can interpret the proof of Theorem \ref{mainThm} as showing that the permutation of $[q]$ defined by $\nu(\cdot)$ is sufficiently pseudorandom that it contains the same increasing subsequences that could be found in a random permutation (with large probability). In connection with this interpretation, we note that Bourgain \cite{B} has shown that the $L^2$ norm of the maximal function of a randomly ordered bounded orthonormal system is at most $\log\log(N)$ (with large probability). Perhaps this is some indication that the correct bound on $\Delta(N)$ may be near $\log\log(N)$, or at least somewhat smaller than the trivial bound of $\log(N)$.
  \item Consider the following `maximal large sieve' inequality

        \[ \sum_{q \leq Q} \; \sideset{}{^*}\sum_{\chi \mod N}  \max_{\ell \leq N} \left|\sum_{\substack{n=M+1 \\ (n,q)=1 }}^{M+\ell} a_n \chi(n) \right|^2 \ll (N+Q^2) \Delta'(N) \sum_{n=M+1}^{M+N} |a_n|^2.\]

        Here $\sideset{}{^*}\sum_{\chi \mod N}$ denotes summation over primitive characters. Without the $\max_{\ell \leq N}$, one may take $\Delta'(N) \ll 1$, which is (one formulation of) the classical large sieve inequality. With the $\max_{\ell \leq N}$, Uchiyama \cite{U} obtained this with $\Delta'(N) \ll \log^2(N)$ using a Rademacher-Menshov-type argument. It is an open problem to determine if the $\log^2(N)$ factor can be reduced or eliminated. This is implicit in \cite{U} and explicitly stated in \cite{M}, see also \cite{H}. Any improvement on this problem would immediately give new bounds on the maximal Barban-Daveport-Halberstam theorem, and any improvement of the form $\Delta(N)':=o(\log(N))$ would give new bounds on the variational Barban-Davenport-Halberstam inequality \cite{L}. The latter has implications on the size of prime gaps in generic arithmetic progressions.

        Due to the averaging over moduli, it will likely be easier to obtain an improved upper bound on $\Delta'(N)$ (than $\Delta(N)$). Indeed, notice that our counterexample to $\Delta(N) \ll 1$ varies with the modulus, and thus it does not seem possible to use a similar approach to disprove $\Delta'(N) \ll 1$.
\end{itemize}

\section{Acknowledgment}The authors would like to thank the referee for a careful reading of the manuscript.

\texttt{A. Lewko, Microsoft Research}

\textit{allew@microsoft.com}
\vspace*{0.5cm}

\texttt{M. Lewko, Department of Mathematics, The University of California at Los Angeles}

\textit{mlewko@math.ucla.edu}

\end{document}